\documentclass[12pt]{amsart}

\usepackage{graphicx}
\usepackage{amscd,amsxtra}
\usepackage{amsfonts}
\usepackage{amsthm}
\usepackage{amssymb,latexsym,amsmath}

\theoremstyle{plain}
\newtheorem{thm}{Theorem}[section]
\newtheorem{theorem}[thm]{Theorem}

\newtheorem{lemma}[thm]{Lemma}

\newtheorem{proposition}[thm]{Proposition}
\theoremstyle{definition}
\newtheorem{remark}[thm]{Remark}

\newtheorem{definition}[thm]{Definition}
\newtheorem{claim}[thm]{Claim}

\newtheorem{question}[thm]{Question}

\numberwithin{equation}{section}

\newcommand{\sO}{{\mathcal O}}

\newcommand{\sV}{{\mathcal V}}

\newcommand{\A}{{\mathbb A}}

\newcommand{\C}{{\mathbb C}}

\newcommand{\BP}{{\mathbb P}}
\newcommand{\Q}{{\mathbb Q}}
\newcommand{\R}{{\mathbb R}}

\newcommand{\Z}{{\mathbb Z}}

\newcommand{\Aut}{{\rm Aut}\,}
\newcommand{\Bir}{{\rm Bir}\,}

\usepackage[all]{xy} 
\usepackage{color}

\title [Primitive automorphisms]{On algebraic dynamics of birational automorphisms of Calabi-Yau manifolds and the universal Coxeter groups}

\author{Keiji Oguiso}

\address{Mathematical Sciences, the University of Tokyo, Meguro Komaba 3-8-1, Tokyo, Japan and National Center for Theoretical Sciences, 
Mathematics Division, National Taiwan University, 
Taipei, Taiwan}
\email{oguiso@ms.u-tokyo.ac.jp}

\thanks{The author is supported by JSPS Grant-in-Aid (A) 25H00587, 25K21992 and NCTS scholar program.}

\begin{document}

\maketitle

\begin{abstract} This is a continuation of our study of algebraic dymanics of birational automorphisms of complex Calabi-Yau manifolds in the view of primitivity and the existence of Zariski dense orbit. The universal Coxeter groups and their Coxeter elements also play important roles in our study. 
\end{abstract}

\section{Introduction}\label{sect1}

Unless stated otherwise, we work in the category of projective varieties defined over an {\it uncountable} algebraically closed field $k$ with a fixed embedding $k \subset \C$. Our main results are Theorems \ref{mainthm1}, \ref{mainthm2}, which are generalizations of our previous works \cite{Og18} and \cite{Og19}, with corrections of some arguments there, and Theorem \ref{thma52} on the universal Coxeter group.

\medskip

Let $V$ be a projective variety of dimension $d \ge 2$ and $f : V \dasharrow V$ be a birational self-map, i.e., $f \in \Bir (V)$. Here and hereafter we denote by $f^{n}$ the $n$th iteration of $f$. Algebraic dynamics studies behaviours of $f^{n}$ under $n \to \infty$. One of the most natural and basic questions is the existence of Zariski dense orbit:

\begin{question}\label{ques1} When does $f$ have a Zariski dense orbit, i.e., when is there $x \in V(k)$ such that the orbit $\{f^n(x)\, |\, n \in \Z_{\ge 0}\}$ of $x$ under $f$ is Zariski dense in $V$? (See Section \ref{sect2} for a more precise formulation.)
\end{question}

By Amerik and Campana \cite[Th\'eor\`eme 4.1]{AC13} (see also \cite[Theorem 1.2]{BGR17} and Theorem \ref{thm21} in Section \ref{sect2}) and by the {\it uncountability assumption} of $k$, we have:

\begin{theorem}\label{thm1}
Let $f \in \Bir (V)$ with ${\rm ord}\, f = \infty$. Then $f$ has a Zariski dense orbit if and only if there is no commutative diagram
$$
  \xymatrix{
    & V \ar@{..>>}[d]^{\varphi} \ar@{..>>}[r]^{f} & V \ar@{..>>}[d]^{\varphi}\\
    & W \ar[r]^{{\rm Id}_W} & W. & 
  }$$
for any projective variety $W$ with $0 < \dim W < \dim V$ (under the uncountability of $k$). Moreover, in this case, the set of $x \in V(k)$ with Zariski dense orbit $\{f^n(x)\}_{n \ge 0}$ is also Zariski dense in $V$.
\end{theorem}

In particular, under the setting of Theorem \ref{thm1}, the existence of a Zariski dense orbit is invariant under a birational modification, that is, if $\nu : V' \dasharrow V$ is a birational map, then $f \in \Bir (V)$ has a Zariski dense orbit if and only if $f' := \nu^{-1} \circ f \circ \nu$ has a Zariski dense orbit.

\begin{remark}\label{rem0} Theorem \ref{thm1} is widely open if $k$ is countable, e.g. if $k = \overline{\Q}$. See Matsuzawa-Wang \cite{MW24}, Matsuzawa-Xie \cite{MX25}, especially Theorem 1.6 there, for important progress.
\end{remark} 

The notion of primitivity of $f \in \Bir (V)$ is introduced by De-Qi Zhang in his study of algebraic dynamics on higher dimensional projective varieties \cite{Zh09} (See also \cite{Og14}):

\begin{definition}\label{Primitive} $f \in \Bir (V)$ is said to be {\it primitive} if there is no commutaive diagram 
$$
  \xymatrix{
    & V \ar@{..>>}[d]^{\varphi} \ar@{..>>}[r]^{f} & V \ar@{..>>}[d]^{\varphi}\\
    & W \ar@{..>>}[r]^{f_W} & W. & 
  }$$
for any projective variety $W$ with $0 < \dim W < \dim V$. In particular, ${\rm ord}\, f = \infty$ if $f$ is primitive.
\end{definition}
Primitivity of $f$ is also invariant under a birational modification in the sense above. By Theorem \ref{thm1}, primitivity of $f$ implies that $f$ has a Zariski dense orbit under the assumption that $k$ is uncountable. We also remark that primitivity of is a kind of {\it irreducibility} in the category of biratioal equivalence classes of projective varieties. 

\medskip

Main objects of this paper is the existence of a Zariski dense orbit and primitivity of a birational automorphism of a Calabi-Yau manifold and a minimal Calabi-Yau variety. See Theorems \ref{thm23}, \ref{thm24} for the reason why Calabi-Yau manifolds and minimal Calabi-Yau varieties are natural and important as target varieties in this study.

\begin{definition}\label{def1} Let $V$ be a projective variety 
of dimension $d$. 
\begin{enumerate}
\item $V$ is said to be a {\it strict Calabi-Yau manifold}, if $V$ is smooth, $H^0(V, \Omega_V^j) = 0$ for $0 < j < d$, $H^0(\Omega_V^{d}) = k \omega_V$ with a nowhere vanishing regular $d$-form $\omega_V$ and the underlying complex manifold $V_{\C}^{{\rm an}}$ is simply-connected.

\item $V$ is said to be a {\it minimal Calabi-Yau variety} if $V$ has at most $\Q$-factorial terminal singularities, $H^1(V, \sO_V) = 0$ and the numerical equivalence class of canonical divisor
$K_V$ is $0$, i.e., $K_V \equiv 0$ in $N^1(V)_{\Q}$. We note that the last condition is equivalent to a stronger condition that $\sO_V(mK_V) \simeq \sO_V$ for some positive integer $m$ (\cite[Theorem1]{Ka13}, see also \cite[Theorem 1.2]{Go13}). 
\end{enumerate}
\end{definition}

Our first main result is the following:

\begin{theorem}\label{mainthm1} Let $V$ be a minimal Calabi-Yau variety of dimension  $d \ge 3$ and $f \in {\rm Bir}\, (V)$. Then $f$ is primitive and $f$ has a Zariski dense orbit, provided that the following two conditions are satisfied:

\begin{enumerate}

\item For any movable effective divisor $D$, there are a minimal Calabi-Yau variety $V'$ and a birational map $\nu : V' \dasharrow V$ such that $D' = \nu^*D$ is semi-ample on $V'$; and 

\item The characteristic polynomial $\Phi_{f}(x)$ of $f^*|_{N^1(V)_{\Q}}$ is either irreducible over $\Q$ or $\Phi_{f}(x) = \varphi(x)(x+a)$ where $\varphi(x) \in \Q[x]$ is irreducible of degree $\ge 2$ and $a \in \Q_{>0}$.
\end{enumerate}
\end{theorem}

\begin{remark}\label{rem1} 

\begin{enumerate}

\item Theorem \ref{mainthm1} is a generalization of \cite[Theorem 1.2]{Og18} {\it with correction of arguments \cite[Claim 2.3 and after that]{Og18}}. We also remark that all results of \cite{KP24} based on the proof of \cite[Theorem 1.2]{Og18} are valid under our correction. Our new argument closely follows \cite[Section 4]{LB19}, which shows that $f \in {\rm Bir}\, (V)$ is primitive if the first dynamical degree $d_1(f) > 1$ when $V$ is a projective hyperk\"ahler manifold. However, contrary to the case of hyperk\"ahler manifolds, there is a strict Calabi-Yau threefold $V$ with automorphism $f \in \Aut (V)$ with $d_1(f) > 1$ but $f$ is not primitive and $f$ has no Zariski dense orbit (See \cite[Example 1, Theorem 1.2]{Og25}).

\item The condition (1) in Theorem \ref{mainthm1} is unconditional if log MMP and log abundance theorem hold in dimension $d$. Note that, since $D$ in (1) is effective and movable, under log MMP for $(V, \epsilon D)$ ($\epsilon$ is a small positive rational number), all extremal contractions involved are small contractions. Thus, under the assumption of log MMP, the output $(V', \epsilon D')$ is obtained by a finite sequence flops under which $\Q$-factorial terminal condition and numerically triviality of the canonical class are preserved and $\epsilon D'\equiv K_{V'} + \epsilon D'$ becomes nef. Then $mD'$ is Cartier and free for some positive integer $m$ if the log abundance conjecture holds for $V'$. See \cite{KM98} and \cite{Ka24} for basic notions and facts on log MMP.     

\item The condition (2) is also invariant under a birational map $\nu : V \dasharrow V''$ between minimal Calabi-Yau varieties. Since $\nu$ is isomorphic in codimension one as $V$ and $V''$ have only $\Q$-factorial terminal singularities with numerically trivial canonical class (see e.g. \cite[Page 240, in the proof]{Ka08}), $\nu^* : N^1(V'')_{\Q} \simeq N^1(V)_{\Q}$ is a well-defined isomorphism of $\Q$-vector spaces and thus $\Phi_{f}(x) = \Phi_{f''}(x)$ where $f'' := \nu \circ f \circ \nu^{-1}$. Here we also note that $(f'')^* = (\nu^*)^{-1} \circ f^* \circ \nu^*$ as again $V$ and $V''$ are $\Q$-factorial and $f$ and $\nu$ are isomorphic in codimension one.

\end{enumerate}
\end{remark}

We shall show Theorem \ref{mainthm1} in Section 4.

\begin{definition}\label{def2} Let 
$$V := V_{d} := (2, 2, \ldots, 2) \subset (\BP^1)^{d+1}$$ 
be a {\it general} hypersurface of multi-degree $(2, 2, \ldots, 2)$ in 
$(\BP^1)^{d+1}$ with $d \ge 3$. More precisely, the word {\it general} here requires the three Zariski dense open conditions under $d \ge 3$ (\cite[Theorem 1.3, and its proof]{CO15}): $V$ is smooth,  $\Aut (V) = \{{\rm Id}_V\}$ and the $d+1$ projections $p_k : V \to P_k$ ($1 \le k \le d+1$) contract no divisor on $V$. Here $P_k$ is the product $(\BP^1)^{d}$ obtained by removing the $k$-th factor of $(\BP^1)^{d+1}$ and $p_k$ is the natural projection. $V$ is a $d$-dimensional strict Calabi-Yau manifold with $\rho(V) = d+1$. We call $V$ a Calabi-Yau manifold {\it of Wehler type}. 
\end{definition}

Our second main result is the following:

\begin{theorem}\label{mainthm2} 
\begin{enumerate}

\item Every Calabi-Yau manifold of Wehler type $V_d$ ($d \ge 3$) has a primitive birational automorphism, hence, a birational automorphism with Zariski dense orbit.

\item Let $X$ be a minimal Calabi-Yau variety of $\dim X = 3$ and of $\rho(X) = 2$. Then the following three statements are equivalent for $f \in \Bir (X)$:

(i) ${\rm ord}\, f = \infty$ (ii) $f$ is primitive (iii) $Z_f \not= \emptyset$.

Moreover, (2) holds also over $k = \overline{\Q}$. 
\end{enumerate}
\end{theorem}

Theorem \ref{mainthm2} is an application of Theorem \ref{mainthm1}. 

\begin{remark}\label{rem2} 
\begin{enumerate}

\item Algebraic dynamics of $f \in {\rm Bir}\, (V)$ of a minimal Calabi-Yau variety $V$ is interesting only when $\rho(V) \ge 2$, as $\Bir (V) = \Aut (V)$ and it is a finite group if $\rho(V) = 1$ (See Section \ref{sect6}). There are strict Calabi-Yau manifolds with $\rho(V) = 2$ and with $f \in \Bir (V)$ of infinite order for any dimension $d \ge 3$. (See e.g. \cite[Theorem 1.4]{Og18}). 
 
\item We will present two proofs of Theorem \ref{mainthm2} (1). Both proofs are based on some remarkable properties of the universal coxeter group (Theorem \ref{thma52}). The first proof is a generalization of \cite[Theorem 8.1]{Og19} with some correction of our previous argument \cite[Page 1395, last 9 lines of the proof of Lemma 8.5]{Og19}. In this proof, $f$ is implicit.  Our second proof uses the Coxeter element of the universal Coxeter group (Section \ref{sect3}) and provides an explicit $f \in \Bir (V)$. This answers a quetion posed by Professor Shigeharu Takayama in his seminar.

\end{enumerate}
\end{remark}

We will show Theorem \ref{mainthm2} (1) in Section \ref{sect5} and Theorem \ref{mainthm2} (2) in Section \ref{sect6}.

\section{Preliminaries on algebraic dynamics}\label{sect2}

In this section, we recall some notions and known results on algebraic dynamics relevant to our main results. 

\subsection{Zariski dense orbit conjecture}\label{sect21}

Let $k$ be an algebraically closed field. In this subsection, we do not assume uncountability of $k$ and also allow $k$ to be of positive characteristic. Let $V$ be a projective variety and $f : V \dasharrow V$ be a dominant rational self-map of $V$ over $k$. Here and hereafter we denote the indeterminacy locus of $f$ by $I(f) \subset V$ and also denote
$$V_f := V_f (k) := \{x \in V(k)\,|\, f^n(x) \not\in I(f)\,\, \forall n \in \Z_{\ge 0}\},$$
$$Z_f := Z_f(k) := \{x \in V_f(k)\,|\, \overline{\{f^n(x)\}_{n \ge 0}} = V\},$$
where $\overline{A}$ is the Zariski closure of $A \subset V$ in $V$. Note that $V_f$ is Zariski dense in $V$ if $k$ is uncountable. The same is true for the countable field $k= \overline{\Q}$ by Amerik \cite[Corollary 9]{Am11}. If $x \in V_f$, then the points $f^n(x)$ are well defined points of $V$ and satisfy $f^{n+m}(x) = f^n(f^m(x))$.

We call a closed point $x \in V(k)$ {\it very general} (resp. {\it general}) if there are countably many {resp. finitely many} proper closed subvarieties $V_i \subset V$ ($i \in I$) such that $x \in V \setminus \cup_{i \in I} V_i$. For instance, every point of $V_f$ is a very general point of $V$. Note that the set of very general points are dense in $V$ if $k$ is uncountable, while it might be empty if $k$ is countable. 

The next theorem due to \cite[Th\'eor\`eme 4.1]{AC13} (see also \cite[Theorem 1.2]{BGR17}) is fundamental to the existence of Zariski dense orbit:

\begin{theorem}\label{thm21} Let $V$ be a projective variety and $f : V \dasharrow V$ be a dominant rational self-map. Then there exist a projective variety $W$ (possibly ${\rm Spec}\, k$) and a dominant rational map $\varphi : V \dasharrow W$ such that $\varphi \circ f = \varphi$ and $\varphi^{-1}(\varphi(x)) = \overline{\{f^n(x)\}_{n \ge 0}}$ for very general points $x \in V(k)$ within $V_f$. 
\end{theorem}

\begin{remark}\label{rem22} 
\begin{enumerate}
\item Theorem \ref{thm1} in Introduction is a direct consequence of Theorem \ref{thm21} when $k$ is uncountable. Thus primitivity of $f$ implies $Z_f \not= \empty$ under the same assumption on $k$.

\item Primitivity is stronger than the existence of Zariski dense orbit even over $\C$, as the following simple example shows.  Let $E$, $F$ be two elliptic curves defined over $\C$, which are not isogenous. Let $a \in E$ and $b \in F$ be non-torsion points. Then the translations $t_{(a, b)}$ and $t_b$ form a commutative diagram 

$
  \xymatrix{
    & E \times F \ar[d]^{{\rm pr}_2} \ar[r]^{t_{(a, b)}} & E \times F \ar[d]^{{\rm pr}_2}\\
    & F \ar[r]^{t_b} & F & 
  }
$ 

and thus $t_{(a, b)}$ is not primitive. On the other hand, since $E$ and $F$ are not isogenous, 
$$\overline{\{t_{(a, b)}^n(x, y)\}_{n \ge 0}} = (x, y) + \overline{\{n(a, b)\}_{n \ge 0}} = E \times F,$$
that is, the orbit of any $(x, y) \in E \times F$ is Zariski dense in $E \times F$. 

\end{enumerate}
\end{remark}

Let $V$ be a smooth projective variety. We denote by $\kappa (V) := \kappa (V, K_V)$ the Kodaira dimension and by $q(V) := \dim H^1(V, \sO_V^1)$ the irregularity of $V$. Set:
$$\sV_{d} := \{(0,0), \ldots, (0, d-1), (0, d), (-\infty,0), \ldots, (-\infty, d-1) \}.$$ Then by \cite[Theorem 1.7]{CLO25}, we have:

\begin{theorem}\label{thm23} Let $V$ be a smooth projective variety of $d:= \dim (V) \ge 2$ and $f \in \Bir (V)$. Then the following assertions hold even over $\overline{\Q}$:

(1) If $Z_f \not= \emptyset$, then $(\kappa (V), q(V)) \in \sV_{d} \setminus \{(0,d-1)\}$. 

(2) Conversely, for each $(\kappa, q) \in \sV_{d} \setminus \{(0,d-1)\}$, there exist $V$ and $f \in \Aut (V)$ such that 
$(\kappa (V), q(V)) = (\kappa, q)$ and $Z_f \not= \emptyset$. 
\end{theorem}

The following theorem due to D.-Q. Zhang \cite{Zh09} (see also \cite[Subsection  2.1]{Og14}) is fundamental: 

\begin{theorem}\label{thm24} Assume that the base field $k$ is an algebraically closed field of characteristic zero. Then a smooth projective variety $V$ with a primitive $f \in \Bir (V)$ satisfies one of the following:

(1) $\kappa (V) = -\infty$ and $q(V) = 0$;

(2) $\kappa (V) = 0$ and $q(V) = 0$; or

(3) birational to an abelian variety ($\Leftrightarrow \kappa (V) = 0$ and $q(V) = \dim\, V$ by \cite[Corollary 2]{Ka81}). 

Assume further that in dimension $d$ both the minimal model conjecture and the weak abundance conjecture that claims $h^0(mK_X) > 0$ for some $m >0$ if $K_X$ is nef hold. Then $V$ of $\dim V = d$ in (1) and (2) are:

(1)' a rationally connected manifold;

(2)' birational to a minimal Calabi-Yau variety. 

\end{theorem}

Theorems \ref{thm23}, \ref{thm24} show that strict Calabi-Yau manifolds and minimal Calabi-Yau varieties are natural and interesting also as target varieties in algebraic dynamics.

\subsection{Dynamical degrees of a dominant sel-fmap}\label{sect23}

In this subsection, we briefly recall the notion of dynamical degrees and their basic properties.  

\begin{definition}\label{def22} Let $V$ be a smooth projective variety of $\dim V = d$. over a field $k$ of characteristic $0$ and $H$ a very ample Cartier divisor on $V$. Then, the  $p$th dynamical degree of a dominant rational sel-fmap $f : V \dasharrow V$ is defined by
$$d_p(f) := \lim_{n \to \infty} ((f^{n})^{*}(H^{p}).H^{d-p})_{V}^{1/n} := \lim_{n \to \infty} (q_n^{*}(H^{p}).p_n^{*}(H^{d-p}))_{V_n}^{1/n} \ge 1,$$
where $p_n : V_n \to V$ is a birational morphism from a smooth projective variety $V_n$ such that $q_n : V_n \to V$ is a morphism and 
$f^n = q_n \circ p_n^{-1}$. 
\end{definition}

Dynamical degrees satisfy the following basic properties (See Dinh-Sibony \cite{DS05} and references therein):

\begin{theorem}\label{thm25}

\begin{enumerate}

\item $d_p(f)$ ($f \in \Bir (V)$) exists and it is independent of the choices involved. Moreover, $d_p(f)$ is a birational invariant in the sense that if $\nu : V \dasharrow V'$ is a birational map, then $d_p(\nu \circ f \circ \nu^{-1}) = d_p(f)$. 
 
\item For a surjective morphism $f : V \to V$ over $\C$, $d_p(f)$ is the spectral radius $r_p(f)$ of $f^{*}|N^p(V)$ and the topological entropy $h_{{\rm top}}(f) \ge 0$ satisfies: 
$$(0 \le) h_{{\rm top}}(f) = {\rm Max}_{p} \log d_p(f) = {\rm Max}_{p} \log r_p(f).$$
\end{enumerate}
\end{theorem}

Salem numbers play important roles in the complex and algebraic dynamics of surface automorphisms and Coxter group theory since the breakthrough works of McMullen \cite{Mc02a}, \cite{Mc02b}. Salem numbers are also important in our study.

\begin{definition}\label{def22} A real algebraic integer $a$ is called a {\it Salem number} if $a > 1$ and its Galois conjugate are $a$, $1/a$ and some complex numbers $e_i$ with $|e_i| = 1$ (possibly empty). We call the minimal monic polynomial $\varphi_a(x) \in \Z[x]$ of a Salem number $a$ the {\it Salem polynomial}  of $a$. We call a monic polynomial $\varphi(x) \in \Z[x]$ a {\it generalized Salem polynomial} if it is the minimal polynomial of $a$ or $-a$ for some Salem number $a$. We note that a generalized Salem polynomial is an irreducible reciprocal polynomial of even degree.
\end{definition}

\begin{proposition}\label{prop23} Let $L = (L, B)$ be an integral hyperbolic lattice of rank $r+1$, that is, $L$ is a free $\Z$-module of rank $r+1$ and $B : L \times L \to \Z$ is a symmetric bilinear form of signature $(1, r)$. Then the characteristic polynomial of $f \in {\rm O}(L, B)$ is the product of at most one generalized Salem polynomial counted with multiplicities and cyclotomic polynomials (possibly $1$). 
\end{proposition}

\begin{proof} Set $L_{\R} = L \otimes_{\Z} \R$ and $B_{\R} = B \otimes_{\Z} \R$. Then the set $\{x \in L_{\R}|B_{\R}(x, x) > 0\} \subset L_{\R}$ consists of two connected components, say $P$ and $-P$. Then $f(P) = P$ or $f(P) = -P$. In the first case, the characteristic polynomial of $f$ is the product of at most one Salem polynomial counted with multiplicities and cyclotomic polynomials (including $1$) (See eg. \cite[Proof of Proposition 7.1]{Mc02b}). In the second case, we have $(-f)(P) = P$ and the result follows from the first case applied for $-f$.
\end{proof}

\section{Universal Coxeter group and Coxeter element}\label{sect3}

The universal Coxeter group and its Coxeter element play crucial role in our proof of Theorem \ref{mainthm2}. Our main result of this section is Theorem \ref{thma52}. Besides its applicability for Theorem \ref{mainthm2}, Theorem \ref{thma52} has its own interest and will be useful in other study. 

Let $M = (m_{ij})$ be a symmetric $(r+1) \times (r+1)$-matrix with $m_{ij} \in \Z_{\ge 2} \cup \{\infty\}$ and $m_{ii} = 1$. We define the group $G = G(M)$ by
$$G := G(M) := \langle s_1, s_2, \cdots, s_{r+1}\,|\, (s_is_j)^{m_{ij}} = 1\rangle.$$ 
The pair $(G(M), S := \{s_i\}_{i=1}^{r+1})$, denoted simply by $G$, is called the {\it Coxeter group} with the Coxeter system $(S, M)$. We call the element 
$$w := w_{G} := s_1s_2 \cdots s_{r+1} \in G$$
the {\it Coxeter element} of $G$. It is known that $s_{\sigma(1)}s_{\sigma(2)} \cdots s_{\sigma(r+1)}$ ($\sigma \in S_{r+1}$) is conjugate to $w$ if the Coxeter graph associated to $G$ is connected and it contains no circle (See e.g. \cite[5.1, 8.4]{Hum90}). 

We define the symmetric bilinear form $B_G$ on $V_G := \R^{r+1} = \oplus_{j=1}^{r+1} \R e_j$ by 
$B(e_i, e_j) = \cos (\pi/m_{ij})$ for all $1 \le i, j \le r+1$, so that $B(e_i, e_i) = -1$ in our sign convention. We denote by  
${\rm O}(V_G, B_G)$ the real orthogonal group of $V_G$ with respect to the bilinear form $B_G$. Then we can define the group homomorphism $\rho : G \to {\rm O}(V_G, B_G)$ by $s_i \mapsto \rho(s_i)$, where
$$\rho(s_i)(x) :=  x + 2B(x, e_i)e_i\,\, (x \in V_G).$$
Our sign convention here fits well with our purpose.
It is wellknown that $\rho$ is a well-defined group homomorphism and it is injective (See e.g. \cite[Page 113, Corollary]{Hum90}). We call the representation $\rho$ the {\it geometric representation} of the Coxeter group $G$.

Benoist and de la Harpe \cite[Th\'eor\`eme, Page 1358]{BH04} proved the following:

\begin{theorem}\label{thma1}
$G$ be the Coxeter group. Assume that the bilinear form $B_G$ on $V_G$ is non-degenerate and indefinite. Then the image $\rho(G)$ of the geometric representation $\rho : G \to {\rm O}(V_G, B_G)$ is Zariski dense in the real algebraic group ${\rm O}(V_G, B_G)$. 
\end{theorem}

The {\it universal Coxeter group} $({\rm UC}(r+1), \{s_i\}_{i=1}^{r+1})$ of rank $r+1$ is the Coxter group with $m_{ij} = \infty$ for all $1 \le i \not= j \le r+1$, i.e., 
$${\rm UC}(r+1) = \langle s_1 \rangle * \langle s_2 \rangle * \cdots * \langle s_{r+1} \rangle \simeq *^{r+1} \Z/2.$$
Let $L := \oplus_{j=1}^{r+1} \Z e_j$. Then the bilinear form $B := B_{{\rm UC}(r+1)}$ of ${\rm UC}(r+1)$ satisfies $B(e_i, e_i) = -1$ and $B(e_i, e_j) = 1$ ($i \not= j$). Therefore, the geometric representation ${\rm UC}(r+1)$ is defined over $\Z$, that is, 
$$\rho : {\rm UC}(r+1) \to {\rm O}(L, B),$$
where ${\rm O}(L, B)$ is the {\it integral} linear orthogonal group of $L \simeq \Z^{r+1}$ with respect to the integral bilinear form $B$ on $L$, and we may regard ${\rm UC}(r+1)$ as a subgroup of ${\rm O}(L, B)$. The signature of $B$ is $(1, r)$ if $r \ge 2$ and therefore $(L, B)$ is a hyperbolic lattice of rank $r+1$ for $r \ge 2$. (See e.g. \cite[2.2.2]{CO15} for the proof.)

We will use the following theorem (Theorem \ref{thma52}) in our proof of Theorem \ref{mainthm2}. To my best knowledge, there is even no literature which states Theorem \ref{thma52}. Since the proof is fairly elementary, we shall also give a proof of Theorem \ref{thma52}. 

\begin{theorem} \label{thma52}
Let $r \ge 2$ and regard ${\rm UC}(r+1) \subset {\rm O}(L, B)$ via the geometric representation $\rho$. Let $\Phi_g(x)$ be the characteristic polynomial of $g \in {\rm UC}(r+1)$ regarded as a linear transformation of $L$. Then: 

\begin{enumerate}
\item There are infinitely many $g \in {\rm UC}(r+1)$ such that $\Phi_g(x)$ is a generalized Salem polynomial of degree $r+1$ when $r$ is odd and there are infinitely many $g \in {\rm UC}(r+1)$ such that $\Phi_g(x)$ is the product of a generalized Salem polynomial of degree $r$ and $x+1$ when $r$ is even.

\item The characteristic polynomial $\Phi_{w}(x)$ of the Coxeter element $w$ of ${\rm UC}(r+1)$ is a Salem polynomial of degree $r+1$ if $r$ is odd and the product of a Salem polynomial of degree $r$ and $x+1$.
\end{enumerate}

\end{theorem}

\begin{proof} First we show the assertion (1). Our proof of (1) is a modification of the proof of \cite[Theorem 4.1]{EOY16}. 

Set $L_{\R} := L \otimes_{\Z} \R \simeq \R^{1+r}$ and $B_{\R} = B \otimes_{\Z} \R$. Since $B_{\R}$ is of signature $(1, r)$. we may choose a real basis
$$\langle v_1, v_2, v_3, \cdots, v_{r+1} \rangle$$ 
of $L_{\R}$ under which $B_{\R}$ is represented by the matrix 
$$Q = {\rm diag}\,(1, -1, -1, \cdots, -1).$$ 
We identify $\R$-linear maps of $L_{\R}$ and their matrix representations via this real basis. 
We set $k = (r+1)/2$ when $r$ is odd and $k = r/2$ when $r$ is even. 
According to $r$ is odd or even, we consider the $\R$-linear map of $L_{\R}$ given by the matrix 
$$M = P \oplus R_2 \oplus \cdots \oplus R_k\,\, {\rm or}\,\, M = P \oplus R_2 \oplus \cdots \oplus R_k \oplus (-1)$$
where 
$$P= \left(\begin{array}{rr}
\sqrt{2} & 1\\
1 & {\sqrt{2}} \end{array} \right)\,\, ,\,\, 
R_i = \left(\begin{array}{rr}
\cos 2\pi\sqrt{2} & -\sin 2\pi\sqrt{2}\\
\sin 2\pi\sqrt{2} & \cos 2\pi\sqrt{2} 
\end{array} \right)\,\, ,$$
for all $2 \le i \le k$. Then $M \in {\rm O}(L_{\R}, B_{\R})$.
 
\begin{lemma} \label{lema1} The characteristic polynomial of $M$ has no cyclotomic factor nor Salem factor if $r$ is odd and the characteristic polynomial of $M$ has no cyclotomic factor nor Salem factor except $x+1$ if $r$ is even.
\end{lemma}

\begin{proof} By the shape of $M$, the eigenvalues of $M$ is not an algebraic integer, except $-1$ when $r$ is even.
\end{proof} 

\begin{lemma} \label{lem52}
There are only finitely many cyclotomic polynomials of degree $\le r+1$.
\end{lemma}

\begin{proof} (cf. \cite[proof of Lemma 4.4]{EOY16}) Let $\varphi(m) = |(\Z/m\Z)^{\times}| = [\Q(\zeta_m): \Q]$ be the Euler function. Since the positive integers $m$ with $\varphi(m) \le r+1$ is finite by the first equality, the result follows.
\end{proof}

Let 
$$u_1(x) = x+1\,\, ,\,\, u_2(x) := x-1\,\, , \cdots\,\, ,\,\, u_{N}(x)$$  
be the cyclotomic polynomials in $\Q[x]$ of degree $\le r+1$, where $N$ is the cardinarity of the cyclotomic polynomials of degree $\le r+1$. 

Let $P(r+1) \subset \R [x]$ be the set of monic polynomials of degree $r+1$. Then $P(r+1)$ is identified 
with $\A_{\R}^{r+1}(\R)$ of the real affine space $\A_{\R}^{r+1}$.
The map  
$${\rm char} : {\rm O}(L_{\R}, B_{\R})(\R) \to P(r+1)\,\, ;\,\, h \mapsto \Phi_{h}(x) := {\rm det} (xI_{r+1} -h)\,\, $$
is given by a morphism ${\rm O}(L_{\R}, B_{\R}) \to  \A_{\R}^{r+1}$ of affine varieties over $\R$. 

Let $P_i := \{p(x) \in P(r+1)\,\, | \,\,  u_i(x) | p(x) \}$
($1 \le i \le N$). Then $P_i$ are proper closed algebraic subsets of $P(r+1)$. According to $r$ is odd or even, we define
$$Q(r+1) := \cup_{i=1}^{N} P_i\,\, ,\,\, Q(r+1) := \cup_{i=2}^{N} P_i.$$ 
Then $Q(r+1)$ is a closed algebraic subset of $P(r+1)$ and thus ${\rm char}^{-1}(Q(r+1))$ is a closed algebraic subset of ${\rm O}(L_{\R}, B_{\R})(\R)$. Moreover ${\rm char}^{-1}(Q(r+1)) \not= {\rm O}(L_{\R}, B_{\R})(\R)$ as $M \in {\rm O}(L_{\R}, B_{\R})(\R)$ in Lemma \ref{lema1} does not belongs to ${\rm char}^{-1}(Q(r+1))$. Hence ${\rm char}^{-1}(Q(r+1))$ is a proper closed algebraic subset of ${\rm O}(L_{\R}, B_{\R})(\R)$.

Let $g \in {\rm UC}(r+1)$. Then the characteristic polynomial $\Phi_g(x)$ is in $\Z[x]$, monic and of degree $r+1$. Moreover, $g$ preserves $B$. Thus, by Proposition \ref{prop23}, $\Phi_g(x)$ is the product of cyclotomic polynomials and at most one generalized Salem polynomial counted with multiplicities. Hence if the assertion (1) would not hold, then $g \in {\rm char}^{-1}(Q(r+1))$ for all $g \in {\rm UC}(r+1)$. Thus the Zariski closure of ${\rm UC}(r+1)$ in ${\rm O}(L_{\R}, B_{\R})$ is in the proper closed algebraic subset ${\rm char}^{-1}(Q(r+1))$, a contradiction to Theorem \ref{thma1}. This completes the proof of the assertion (1) in Theorem \ref{thma52}.

\medskip

Next we show the assertion (2) in Theorem \ref{thma52}. Let 
$$w = s_1s_2 \cdots s_{r+1} \in {\rm UC}(r+1)$$ be the Coxeter element of ${\rm UC}(r+1)$ and let $\langle e_i \rangle_{i=1}^{r+1}$ be the $\Z$-basis of $L$ in the integral geometric representation of ${\rm UC}(r+1)$.

First we observe the following:

\begin{lemma}\label{lem51} With respect to the $\Z$-basis $\langle e_i \rangle_{i=1}^{r+1}$ of $L$, the matrix representation of the Coxeter element $w \in {\rm UC}(r+1)$ on $L$ is
$$-U \cdot U^{t}.$$ 
Here $U = (u_{kl})$ is the $(r+1) \times (r+1)$ upper triangle matrix such that $u_{kk} = 1$, $u_{kl} = -2$ for $l > k$ and $u_{kl} = 0$ for $l < k$ and $U^{t}$ is the transposed matrix of $U$.
\end{lemma} 

\begin{proof} We can deduce this as a special case of a general result \cite[Theorem 2.9]{La08}. We may also check directly as follows.

Let $x \in L$ and write $x = \sum_i x_ie_i$ and $y = w(x) = \sum_i y_ie_i$. Then, by
 $$s_j(e_j) = -e_j\,\, ,\,\, s_j(e_i) = e_i + 2e_j\,\, (i \not= j),$$
Then the coefficient $z_i$ of $e_i$ in $z = \sum_i z_ie_i$ remains the same for $s_j(z)$ for any $i \not= j$. Hence, applying $s_{r+1}$, $s_r$, $\cdots$, $s_j$ inductively, we obtain that
$$y_j = -x_j + 2\sum_{i<j}x_i + 2\sum_{i>j}y_i\,\, ,\,\, {\rm i.e.}\,\, ,\,  y_j-2\sum_{i>j}y_i=-x_j+2\sum_{i<j}x_i. $$

In the second formula, the left side is the $j$th entry of $Uy$, and the right side is the $j$th entry of $-U^{t}x$. Thus
$$Uy =-U^{t}x \quad\text{for every }x \in L$$
Since, $\det U = 1 \not= 0$, it follows that $w = -U^{-1}U^{t}$ as matrices with respect to the basis $\langle e_i \rangle_{i=1}^{r+1}$. 
\end{proof}

\begin{lemma}\label{lem52} The characteristic polynomial $\Phi_w(x)$ of $w$ is
$$\Phi_w(x) = \frac{x(x+3)^{r+1}-(3x+1)^{r+1}}{x-1}.$$
\end{lemma} 
\begin{proof} This is true for $r=1$. We prove by induction of $r$. We write the $(r+1) \times (r+1)$ matrix $U$ by $U_{r+1}$. By the definition and by $\det U_{r+1} = 1$, we have by Lemma \ref{lem51} that
$$\Phi_w(x) = \det (xI + U_{r+1}^{-1}U_{r+1}^{t}) = \det (xU_{r+1} + U_{r+1}^{t}).$$
Set $D_{r+1} := (d_{kl}) = xU_{r+1}+U_{r+1}^{t}$. Then $d_{kk} = x+1$, $d_{kl} = -2x$ for $k < l$ and $d_{kl} = -2$ for $k > l$. By subtracting the second line from the first line of $D_{r+1}$ and then by expanding with respect to the first line, we obtain
$$\det D_{r+1} = (x+3) \det D_{r} + (3x+1) \det E_{r},$$
where $E_{r}$ is the $r \times r$ matrix obtained by getting rid of the first line and the second colum of $D_{r+1}$. Then, for $E_r$, by subtracting the second line from the first line and then by expanding with respect to the first line, and continuing this process, we readily find that $\det E_{r} = -2(3x+1)^{r-1}$. 
Substituting this and the induction hypothesis, we obtain
$$\Phi_w(x) = \det D_{r+1} = 
(x+3)\frac{x(x+3)^{r}-(3x+1)^{r}}{x-1} - 2(3x+1)^r$$
$$ = \frac{x(x+3)^{r+1}-(3x+1)^{r+1}}{x-1},$$
as claimed. 
\end{proof}

\begin{lemma}\label{lem53} Let $r \ge 2$ and $w$ be the Coxeter element of ${\rm UC} (r+1)$. Then the characteristic polynomial $\Phi_w(x)$ of $w$ is a Salem polynomial of degree $r+1$ when $r$ is odd and $\Phi_w(x)$ is the product of a Salem polynomial of degree $r$ and $x+1$ when $r$ is even. 
\end{lemma} 

\begin{proof} Let $r \ge 2$ and set $f(x) := x(x+3)^{r+1}-(3x+1)^{r+1}$. Note that $f(1) = 0$ and thus $\Phi_w(x) = f(x)/(x-1)$ (Lemma \ref{lem52}) is certainly a monic polynomial in $\Z[x]$. Since 
$$(\frac{7}{5})^{r+1} \ge (\frac{7}{5})^3 = \frac{343}{125} >  2$$
it follows that 
$$f(2) = 2 \cdot 5^{r+1} - 7^{r+1} = 5^{r+1}(2- (\frac{7}{5})^{r+1}) < 0 
\,\, ,\,\, \lim_{x \to +\infty} f(x) = +\infty.$$
Thus $f(x)$, hence $\Phi_w(x)$, has a real root in $(2, \infty)$. Hence by Proposition \ref{prop23}, $\Phi_w(x)$ has exactly one Salem factor. Compute that  
$$f'(1) =  -4^{r}r \not= 0$$
for all $r \ge 2$, 
$$f(-1) = -2^{r+2} \not= 0$$
if $r$ is odd and 
$$f(-1) = 0\,\, ,\,\, f'(-1) = -(2r+1)2^{r+1} \not= 0$$
if $r$ is even. 

By Proposition \ref{prop23}, we are now reduced to show that $f(\zeta) \not= 0$ for any root of unity $\zeta \not= \pm 1$. Assume to the contrary that $f(\zeta) = 0$ for some $m$th root of unity $\zeta$ with $m \ge 3$. Since $f(x) \in \Z[x]$, it follows that $f(e^{2\pi i/m}) = 0$ as well, as $\zeta$ and $e^{2\pi i/m}$ are Galois conjugate over $\Q$. Hence we may and will assume that $\zeta = e^{2\pi i/m}$ and $f(\zeta) = 0$. By the shape of $f(x)$, we have then
$$\zeta(\zeta+3)^{r+1}= (3\zeta+1)^{r+1}\,\, ,{\rm i.e.},\,\, \zeta = (\frac{3\zeta +1}{\zeta +3})^{r+1}.$$
Therefore 
$$\tau := \frac{3\zeta +1}{\zeta +3}$$
is also a root of unity.
Notice that the triangle with vertices $0$, $3\zeta$, $3\zeta +1$ and the triangle with vertices $O := 0$, $A := 3$, $B := \zeta +3$ in the plane $\C$ are congruent and thus
$$0 < \beta := {\rm arg}\, \tau = \frac{2\pi}{m} - 2 \alpha < \frac{2\pi}{m},$$
where $\alpha = \angle AOB$. Let $C := 1$ and $D := \zeta + 1$. Then 
$$\overline{OD} = 2 \cos (\pi/m) < 2 = \overline{BD}$$
and thus 
$$\frac{\pi}{m} - \alpha = \angle DOB > \angle DBO = \alpha\,\, ,\,\, {\rm i.e.}\,\, ,\,\, \frac{\pi}{m} > 2\alpha.$$ 
Therefore
$$\frac{2\pi}{m} > {\rm arg}\, \tau = \beta = \frac{2\pi}{m} - 2 \alpha > \frac{2\pi}{m} - \frac{\pi}{m} = \frac{\pi}{m},$$
and hence
$$\frac{\pi}{m} < {\rm arg}\, \tau < \frac{2\pi}{m}.$$
However, this is impossible, as $\tau$ is a root of unity in $\Q(e^{2\pi i/m})$ so that $\tau^{2m} =1$, hence, ${\rm arg}\, \tau = 2k\pi/2m = k\pi/m$ for some integer $k$.\end{proof}

This completes the proof of Theorem \ref{mainthm2} (2). \end{proof}

\section{Proof of Theorem \ref{mainthm1}}\label{sect4}

In this section, we prove Theorem \ref{mainthm1}. Throughout this section, for a dominant rational map $\pi : X \dasharrow B$ with connected fibers between projective varieties, which we call a {\it rational fibration}, we choose a resolution $\nu : \tilde{X} \to X$ of the singularities of $X$ and $I(\pi)$ so that $\tilde{\pi} = \pi \circ \nu : \tilde{X} \to B$ is a surjective morphism with connected smooth general fibers. We denote by $\tilde{X}_b$ the fiber over $b \in B$ of $\tilde{\pi}$ and set $X_b := \nu(\tilde{X}_b)$. We call a rational fibration $\pi : X \dasharrow B$ is {\it $f$-equivariant} for $f \in \Bir (X)$ if there is $f_B \in \Bir (B)$ such that $\pi \circ f = f_B \circ \pi$. 

The following important result is due to Bianco \cite[Proposition]{LB19} (see also \cite{Og18}): 

\begin{theorem}\label{thm31} Let $X$ be a projective variety and $f \in {\rm Bir}\, (X)$. Assume that $X$ admits an $f$-equivariant 
rational fibration $\pi : X \dasharrow B$ such that $\dim B < \dim X$ and $\tilde{X}_b$ is of general type for general $b \in B$. Then $Z_f = \emptyset$, i.e., $f$ has no Zariski dense orbit. 
\end{theorem}

In what follows, $V$ is a minimal Calabi-Yau variety satisfying the conditions (1) and (2) in Theorem \ref{mainthm1}. Following the convention above, for a rational fibration $\pi : V \dasharrow B$, we choose a resolution $\nu : \tilde{V} \to V$ of the singularities of $V$ and $I(\pi)$ so that $\tilde{\pi} = \pi \circ \nu : \tilde{V} \to B$ is a surjective morphism with connected smooth general fibers. 
\begin{lemma}\label{lem31} $V$ has no $f$-equivariant rational fibration $\pi : V \dasharrow B$ such that $0 < \dim B < \dim V$ and $\kappa(\tilde{V}_b) = 0$ for general $b \in B$. 
\end{lemma}

\begin{proof} Assuming to the contrary that $V$ has an $f$-equivariant rational fibration $\pi : V \dasharrow B$ such that $0 < \dim\, B < \dim\, V$ and $\kappa(\tilde{V}_b) = 0$ for general $b \in B$, we shall derive 
a contradiction. {\it Our proof of Lemma \ref{lem31} until the end of Calim \ref{cl31} is almost identical to the relevant part of the proof of \cite[Theorem 1.2]{Og18}. However, since we argue by contradiction, to avoid any confusion, we also repeat the identical part.} 

By taking a resolution of singularities of $B$, we may and will assume that $B$ is smooth. Let $H$ be a very ample divisor on $B$. Set $L := \nu_*(\tilde{\pi}^*H)$. Then $|L|$ is movable and $\pi = \Phi_{\Lambda}$ for a sublinear system $\Lambda \subset |L|$ and $\Phi_{\Lambda}$ is the rational map associated to $\Lambda$. Then, by the assumption (1) in Theorem \ref{mainthm1}, there are a minimal Calabi-Yau variety $V'$ and a birational map $g : V' \dasharrow V$ such that $g^*L$ is semi-ample. So, by replacing $(V, f)$ by 
$(V', f')$ with $f' := g^{-1} \circ f \circ g$, we may and will assume that $L$ is semi-ample. By Remark \ref{rem1} (3), the characteristic polynomial of $(f')^*|_{N^1(V')_{\Q}}$ is the same as the characteristic polynomial of $f^*|_{N^1(V)_{\Q}}$.  
Let us take a sufficiently large positive integer $m$ such that $|mL|$ is free and the morphism $\varphi_m := \Phi_{|mL|}$ is the Iitaka-Kodaira fibration associated to $L$. Set $B' = \varphi_m (V)$. Then $B'$ is normal, projective and $\dim\, B' = \kappa(V, L)$, the Iitaka-Kodaira dimension of $L$. As $\Lambda \subset |L| \subset |mL|$, the surjective morphism $\varphi_m : X \to B'$ factors through $\pi : X \dasharrow B$, and hence, there is a dominant rational map $\rho : B' \dasharrow B$ such that $\pi = \rho \circ \varphi_m$. 

\begin{claim}\label{cl31}
$\rho : B' \dasharrow B$ is a birational map.
\end{claim}
\begin{proof}
As $\pi$ is of connected fibers and $\rho$ is dominant, it suffices to show that $\dim B = \dim B'$. Set $\tilde{L} := \nu^*L$. Since $\dim B' = \kappa(V, L) = \kappa(\tilde{V}, \tilde{L})$, it suffices to show that $\dim B = \kappa(\tilde{V}, \tilde{L})$.

Let $\{E_i\}_{i=1}^{m}$ be the set of exceptional prime divisors of $\nu$. Choose a positive integer $n$ such that $\sO_V(nK_{V}) \simeq \sO_V$. Then, in ${\rm Pic}\,(\tilde{V})$, we have
$$\tilde{L} = \nu^*L = \tilde{\pi}^*H + \sum_{i=1}^{m} a_iE_i\,\, ,\,\, nK_{\tilde{V}} = \sum_{i=1}^{m} b_iE_i\,\, ,$$
where $a_i \ge 0$ as $\nu$ resolves $I(\pi)$ and $b_i > 0$ as $V$ is a minimal Calabi-Yau variety. 
Since $\tilde{L}|_{\tilde{V}_b} = \sum_{i=1}^{m} a_iE_i|_{\tilde{V}_b}$ as $H|_{\tilde{X}_b}$ is trivial and $nK_{\tilde{V}_b} = b_i E_i|_{\tilde{V}_b}$ by the adjunction formula. Then there is a positive integer $M$ such that $M K_{\tilde{V}_b} - \tilde{L}|_{\tilde{V}_b}$ is linearly equivalent to an effective divisor. 
Hence, for general $b \in B$, we have by $\kappa(\tilde{X}_b) = 0$ that
$$0 \le \kappa(\tilde{V}_b, \tilde{L}|_{\tilde{V}_b}) \le \kappa(\tilde{V}_b) = 0\,\, .$$
Hence $\kappa(\tilde{V}_b, \tilde{L}|_{\tilde{V}_b}) = 0$ and therefore
$$\dim B  \le \kappa(\tilde{V}, \tilde{L}) \le \dim B + \kappa(\tilde{V}_b, \tilde{L}|_{\tilde{V}_b}) = \dim B\,\, .$$
Here the second inequality follows from \cite[Theorem 5.11]{Ue75}. 
Hence $\dim B = \kappa(\tilde{V}, \tilde{L})$ and the result follows.  
\end{proof}

By Claim \ref{cl31}, we may and will assume that $\pi : V \to B$ is an $f$-equivariant surjective {\it morphism} given by the free complete linear system $|L|$, by replacing $\pi : V \dasharrow B$ by $\varphi_m : V \to B'$ for some $m \ge 1 $. 

As in \cite{LB19}, we consider the following subspace of the $\Q$-vector space $N^1(V)_{\Q}$:
$$M := \Q \langle (\varphi \circ \pi)^*N^1(B)_{\Q}\, |\, \varphi \in {\rm Bir}\, (B) \rangle \subset N^1(V)_{\Q}.$$ 

\begin{claim}\label{cl32}
$0 \not= M \not= N^1(V)_{\Q}$. Moreover $f^*M = M$.
\end{claim}

\begin{proof} Since $0 \not= \pi^*N^{1}(B)_{\Q} \subset M$, we have $M \not= 0$. Let $\varphi \in {\rm Bir}\, (B)$. Then there are Zariski dense subsets $U, U' \subset B$ such that $\varphi|_{U} : U \to U'$ is an isomorphism. Since $\pi$ is a surjective morphism, it follows that $\varphi \circ \pi$ is a surjective morphism from $\pi^{-1}(U) \subset V$ onto $U'$, that is, $\varphi \circ \pi$ is an {\it almost holomorphic map} from $V$. Let $v \in N^1(B)_{\Q}$ and choose a positive integer $m$ and very ample general divisors $h_1$ and $h_2$ on $B$ such that $mv = h_1 -h_2$ in $N^1(B)$. Then 
$$(\varphi \circ \pi)^* (mv) = (\varphi \circ \pi)^* (h_1) - (\varphi \circ \pi)^* (h_2)$$
in $N^1(V)_{\Q}$ and $((\varphi \circ \pi)^* (h_i).C)_V = 0$ ($i=1$, $2$) for a general curve $C$ in a general fiber of almost holomorphic $\varphi \circ \pi$. Thus $M$ contain no ample class of $V$. Hence $M \not= N^1(V)_{\Q}$. Since $\pi \circ f = f_B \circ \pi$ and $f$ is isomorphic in codimension one and $\pi$ is a morphism, we have 
$$f^*((\varphi \circ \pi)^*v) = (\varphi \circ \pi \circ f)^*v = (\varphi \circ f_B \circ \pi)^*v.$$
Since $\varphi \circ f_B \in {\rm Bir}\, (B)$, it follows that $f^*M \subset M$. Hence $f^*M = M$ as $f^*$ is an isomorphism on $N^1(V)_{\Q}$. This completes the proof. \end{proof}

Now we are ready to derive a contradiction to complete the proof of Lemma \ref{lem31}. By the assumption (2) in Theorem \ref{mainthm1}, $M$ is either diminsion one or codimension one and the characteristic polynomial of $f^*|_{N^1(V)_{\Q}}$ is of the form $\varphi(x)(x+a)$, where $\varphi(x) \in \Q[x]$ is irreducible of degree $\ge 2$ and $a \in \Q_{>0}$. 

If $\dim M = 1$, then $M$ is generated by the class of $\pi^*h$ where $h$ is a very ample divior on $B$ and the action of $f^*$ on $M$ is the multiplication by $-a < 0$. Since the divisor $\pi^*h$ is non-zero effective, so is $f^*(\pi^*h)$. On the other hand, $f^{*}(\pi^*h) = -a \pi^*h$ in $N^1(V)_{\Q}$. Let $H$ be a very ample divisor on $V$. We have then  
$$0 < (f^{*}(\pi^*h).H^{d-1})_V = -a ((\pi^*h).H^{d-1})_V < 0,$$
 a contradiction.

Assume that $M$ is of codimension one. By assumption (2) in Theorem \ref{mainthm1}, we have an $f^*$-stable decomposition $N^1(V)_{\Q} = M \oplus \Q l$ with $f^*l = -al$. Note that any element of $M$ is represented by a sum of divisors not dominating $B$ by Claim \ref{cl32} or its proof. Thus, for any curve $C$ in a general fiber of $\pi$, we have $(m.C) = 0$ for all $m \in M$. Thus $(l.C) \not= 0$. By replacing $l$ by $-l$ if necessary, we may and will assume that $(l. C) > 0$. Let $H$ be a very ample divisor on $V$ and write $[H] = m + cl$ with $m \in M$ and $c \in \Q$. Then, by taking the intersection with $C$, we have $c > 0$. On the other hand, as $C$ is general and $f^*H$ is movable, we have also
$$0 \le (f^{*}H.C) = -ac(l.C) < 0,$$
a contradiction. 

This completes the proof of Lemma \ref{lem31}. 
\end{proof}

\begin{lemma}\label{lem32} $Z_f \not= \emptyset$ for $(V, f)$ in Theorem \ref{mainthm1}.
\end{lemma}

\begin{proof} Otherwise, we have a rational dominant map $\pi : V \dasharrow \BP^1$ such that $\pi \circ f  = \pi$ by Theorem \ref{thm1}, by taking a non-constant rational function on $W$ ($\dim W > 0$) in Theorem \ref{thm1}. Here we used the uncountability assumption of the base field $k$. Then $f^*V_b = V_b$ as divisors for general $b \in \BP^1$ as $f$ is isomorphic in codimension one. Since $V$ is $\Q$-factorial, it follows that $1$ is an eigenvalue of $f^*|_{N^1(V)_{\Q}}$, a contradiction to the assumption (2) in Theorem \ref{mainthm1}.
\end{proof}

We are now ready to prove Theorem \ref{mainthm1}. 

Assuming that $V$ admits an $f$-equivariant rational dominant map $\pi : V \dasharrow B$ with $0 < \dim B < \dim V$, we shall derive a contradiction. We may and will assume that $\pi$ is a rtional fibration by taking the Stein factorization. 

Choose a smooth Zariski dense open subset $B_0 \subset B$ such that the induced morphism $\tilde{V}_0 := \tilde{\pi}^{-1}(B_0) \to B_0$ is smooth. 
Consider the relative canonical model of $\tilde{V}_0$ over $B_0$ (\cite{BCHM10}):
$$\varphi_0 : \tilde{V}_0 \dasharrow K_0 := {\rm Proj} \oplus_{m \ge 0}\tilde{\pi}_* \sO_{\tilde{V_0}}(mK_{\tilde{V_0}/B_0})\,\, .$$
For a projective compactification $K$ of $K_0$, the rational fibration $\varphi_0 : V_0 \dasharrow K_0$ induces a rational fibration 
$\varphi : \tilde{V} \dasharrow K$ with a resolution $\tilde{\tilde{V}} \to \tilde{V}$ of $I(\varphi)$. Then $\varphi$ is $f$-equivariant and $\kappa(\tilde{\tilde{V}}_k) = 0$ for general $k \in K$. Here $\tilde{\tilde{V}}_k$ is the fiber of $\tilde{\tilde{V}} \to K$ over $k \in K$. We have $0 < \dim B \le \dim K$ by the construction. On the other hand, by Lemma \ref{lem32} and Theorem \ref{thm31}, $\tilde{V}_b$ is not of general type for general $b \in B$. Thus $\dim K_b < \dim V_b$ for general $b \in B$. Hence $0 < \dim K < \dim V$. However, this contradicts to Lemma \ref{lem31}, as $\kappa(\tilde{\tilde{V}}_k) = 0$ for general $k \in K$. Thus $f$ is primitive. This completes the proof of Theorem \ref{mainthm1}.

\section{Proof of Theorem \ref{mainthm2} (1)}\label{sect5}

In this section, we prove Theorem \ref{mainthm2} (1). 

\medskip

Let $d \ge 3$ and $V$ be a $d$-dimensional Calabi-Yau manifold of Wehler type (Definition \ref{def2}). We will check that $V$ satisfies the condition (1) in Theorem \ref{mainthm1} and $V$ has a birataional automorphism $f \in {\rm Bir}\, (V)$ satisfying the condition (2) of Theorem \ref{mainthm1}. We also use Theorem \ref{thm51} below.

For the statement of Theorem \ref{thm51}, we recall:
$$N^1(V) \simeq {\rm Pic}\, V = \oplus_{k=1}^{d+1} \Z h_k \simeq \Z^{d+1}
\,\, , \,\, \overline{{\rm Amp}}\,(V) = \oplus_{k=1}^{d+1} \R_{\ge 0} h_k\,\, .$$
Here $h_k$ is the pullback of the hyperplane class of $\BP_k^1$ under the natural projection $V \to \BP_k^1$. Let $f \in {\rm Bir}\, (V)$. Then, by Remark \ref{rem1} (3), $f : V \dasharrow V$ is isomorphic in codimension one and we have a natural contravariant representation 
$$\rho^{*} : {\rm Bir}\, (V) \to {\rm GL}\, (N^1(V)) \simeq {\rm GL}\, (\Z^{d+1})\,\, ; \,\, f \mapsto f^*\,\, .$$
As $V$ is smooth, we do not need to replace $N^1(V)$ by $N^1(V)_{\Q}$. 

Let $p_k : V \to P_k$ ($1 \le k \le d+1$) be the natural projection, where $P_k$ is the product manifold $(\BP^1)^{d}$ obtained by removing the $k$-th factor of $(\BP^1)^{d+1}$. The morphism $p_k$ is of degree $2$ and the covering involution $\iota_k$ of $p_k$ is a birational automorphism of $V$ of order two. $p_k$ is never finite as $d \ge 3$. Then by our definition of Wehler type (Definition \ref{def2}), the Stein factorization $\overline{p}_k : V \to V_k$ of $p_k : V \to P_k$ is a {\it small} contraction of relative Picard number $1$. Moreover, $\iota^*h_k + h_k$ is invariant under $\iota_k$ so that $\iota^*h_k + h_k = p_k^*l_k$ for some divisor on $P_k$. In particular, $h_k$ is $\overline{p}_k$-ample and thus $\iota_k^*h_k$ is $\overline{p}_k$-anti-ample and therefore $\iota_k : V \dasharrow V$ is a flop in the sense of MMP. The following theorem is proved by \cite[Theorems 1.3, 3.3]{CO15}:

\begin{theorem} \label{thm51} Let $V$ be a Calabi-Yau manifold of Wehler type of dimension $d \ge 3$. Then:

\begin{enumerate}

\item ${\rm Bir}\, (V) = \langle \iota_1, \ldots, \iota_{d+1} \rangle$ and $({\rm Bir}\, (V), \{\iota_k\}_{k=1}^{d+1})$ forms a Coxeter group isomorphic to the universal Coxeter group ${\rm UC}(d+1)$ of rank $d+1$. 

\item The representation $\rho^{*}$ is faithful and coincides with the dual of the geometric representation of the universal Coxter group $({\rm Bir}\, (V), \{\iota_k\}_{k=1}^{d+1}) \simeq {\rm UC}(d+1)$: for each $1 \le j \le d+1$,
$$\rho^*(\iota_j)(h_j) = -h_j + \sum_{k \not= j}2h_k\,\, ,\,\, \rho^{*}(\iota_j)(h_i) = h_i\,\, (i \not= j).$$ 

\item For any effective movable divisor $D$, there are a nef divisor $H$ and $g \in {\rm Bir}\, (V)$ such that $g^*D = H$ in $N^1(V)$. In particular, $H$ is semi-ample. (See the description of the nef cone $\overline{{\rm Amp}}\,(V)$ above.) 
\end{enumerate}
\end{theorem}

We are ready to prove Theorem \ref{mainthm2} (1). $V$ satisfies the condition (1) in Theorem \ref{mainthm1} by Theorem \ref{thm51} (3). Since $\rho^{*}$ is the dual of the geometric representation of $\rho$ by Theorem \ref{thm51}, $V$ satisfies the condition (2) in Theorem \ref{mainthm1} by Theorem \ref{thma52}(1). Moreover, the image of $\iota_{d+1} \iota_d \cdots \iota_1$ under $\rho^*$ satisfies the condition (2) in Theorem \ref{mainthm1} by Theorem \ref{thma52}(2). Note that the characteristic polynomial remains the same under taking the transposed matrix. Hence $\iota_{d+1} \iota_d \cdots \iota_1$ is a primitive birational automorphism of $V$. This completes the proof of Theorem \ref{mainthm2} (1) with an explicit primitive $f \in \Bir (V)$.

\section{Proof of Theorem \ref{mainthm2} (2)}\label{sect6}

In this section, we prove Theorem \ref{mainthm2} (2). First, we show the following general result.

\begin{proposition} \label{prop61} Let $k$ be an algebraically closed field of characteristic zero. Let $V$ be a minimal Calabi-Yau variety of dimension $d \ge 3$ and $f \in \Bir (V)$ defined over $k$. Set $f^* := f^*|_{N^1(V)_{\Q}}$. Then:

\begin{enumerate}

\item ${\rm ord}\, f = \infty$ if and only if ${\rm ord}\, f^* = \infty$

\item Assume that $\rho(V) = 1$. Then $f \in \Aut (V)$ and ${\rm ord}\, f< \infty$.

\item Assume that $\rho(V) = 2$. Then if ${\rm ord}\, f = \infty$, then the characteristic polynomial $\Phi_f(x) \in \Q[x]$ is irreducible over $\Q$. Moreover, $d_1(f) >1$. 
\end{enumerate}
\end{proposition}

\begin{proof} Let us show (1). It suffices to show that if ${\rm ord}\, f^* < \infty$, then ${\rm ord}, f < \infty$. Since $H^0(X, T_X) \simeq H^1(X, \sO_X)^{*} = 0$ for a minimal Calabi-Yau variety by \cite[Proposition 4.3]{Xu20}, it follows that the identity component $\Aut^0(V)$ is trivial. Thus $f$ is also of finite order by the Fujiki-Lieberman type theorem (See e.g. \cite[Theorem 1.4]{Li20} for singular projective varieties we need).

Let us next show (2). Since $f : V \dasharrow V$ is isomorphic 
in codimension one and preserves the ample cone by $\rho (V) =1$, it follows that $f \in \Aut(V)$ by \cite[Lemma 1.5]{Ka97}. Then $f^*(h) = h$ for the primitive integral ample class $h \in N^1(V)$ as $f \in \Aut(V)$. Thus $f^*|_{N^1(V)_{\Q}} = {\rm id}$ again by $\rho(V) = 1$. Hence the result follows from (1).

Let us show (3). Since $\Phi_{f}(x) \in \Q[x]$ is of degree 2, the irreducibility of $\Phi_{f}(x)$ is equivalent to the fact that $\Phi_{f}(x)$ has no rational zero. Thus, it suffices to show that $\Phi_{g}(x)$ has no rational zero for $g := f^2$. Note that ${\rm ord}\, g^* = \infty$ by (1).

Consider the boundary rays of $\overline{{\rm Mov}}\, (X)$. Since $\rho(X) = 2$ and $\overline{{\rm Mov}}\, (X)$ is strictly convex, there are exactly two such rays, say, $\R_{>0}v_1$ and $\R_{>0}v_2$, and $v_1$ and $v_2$ form $\R$-basis of $N^1(X)_{\R}$. Since $g = f^2$ and $f$ is isomorphic in codimension one, we have $g^* = (f^*)^2$ and. Then $g^*v_1 = a_1v_1$ and $g^*v_2 = a_2v_2$ for some $a_1 > 0$ and $a_2 > 0$ and $\Phi_g(x) = (x-a_1)(x - a_2)$ in $\R[x]$. 

Assume to the contray that one of $a_i$ is rational. Then $a_1, a_2 \in \Q$.

Let $H$ be a very ample Cartier divisor on $V$.

If $a_1 = a_2$, say, $a > 0$, then for any ample integral Weil divisor $D$, we have $(g^{n})^{*}D = a^{n}D$ for all $n \in \Z$, and this is again an integral Weil divisor class as $g$ and $g^{-1}$ are isomorphic in codimension one. Then 
$$0 < a^{n}(D.H^2)_V = ((g^{n})^{*}D .H^2)_V \in \Z$$
for all integers $n$. Then, by considering $n \to \pm \infty$, we see that $a= 1$. Then $g^* = {\rm id}$, a contradiction. 

Now we may assume that $a_1 \not= a_2$. Then we may and will assume that $a_1 \not= 1$. Since $\R v_1$ is the eiganspace of $g^*$ with eigenvalue $a_1$, which is rational and of multiplicity one, we may choose an integral Weil divisor $D$ so that $\R_{>0} v_1 = \R_{>0} [D]$. Since $(v_1.H^2)_V  \ge 0$ and since the class $H^2$ is in the interior of $\overline{{\rm NE}}(X)$, we have $(v_1H^2)_V \not= 0$ and thus $(v_1.H^2)_V > 0$. Thefore $(D.H^2)_V >0$ as well. Again as before, $(g^{n})^{*}D$ is an integral Weil divisor as a divisor. It then 
follows that 
$$0 < a_1^{n}(D.H^2)_V = ((g^{n})^{*}D .H^2)_V \in \Z$$
for all integers $n$. Then, by considering $n \to \pm \infty$ again, we see that $a_1 = 1$, a contradiction to our choice of $a_1$. 

Hence both $a_1$ and $a_2$ are irrational and so are $\sqrt{a_1}$ and $\sqrt{a_2}$. Thus $d_1(f) \not= 1$, as $d_1(f)$ is either $\sqrt{a_1}$ or $\sqrt{a_2}$ as $(f^n)^* = (f^n)^*$. Since $d_1(f) \ge 1$, it follows that $d_1(f) > 1$. This completes the proof of (3). \end{proof}

Now we are ready to prove Theorem \ref{mainthm2} (2). 

Let $X$ be a $3$-dimensional minimal Calabi-Yau variety with $\rho (X) =2$ and $f \in \Bir (X)$. Since $k$ is uncountable or $k = \overline{\Q}$, it follows that (ii) implies (iii) by Theorem \ref{thm1} and \cite[Theorem 1.6]{MX25} when $k = \overline{\Q}$. Note that, by Proposition \ref{prop61} (3), we have $d_1(f) > 1 = d_3(f)$ if $f$ is primitive so that ${\rm ord}\, f = \infty$. It is clear that (iii) implies (i). It remains to show that (i) implies (ii). 

Assume that ${\rm ord}\, f = \infty$. We are going to show that $f$ is primitive by checking the two conditions (1) and (2) in Theorem \ref{mainthm1}. 

By the log abundance theorem for threefold \cite{KMM94}, $X$ satisfies the condition (1) in Theorem \ref{mainthm1}, and by Proposition \ref{prop61} (3), $f$ satisfies the condition (2) in Theorem \ref{mainthm1}. This completes the proof of Theorem \ref{mainthm2} (2).

\section*{Acknowledgements} First of all, I would like to express my thanks to Professor Shigeharu Takayama for his invitation to talk at his seminar and useful discussions there. I noticed and repaired some errors in the proof of my prvious works during preparation of my talk there. I used Chat GPT plus in our study of the universal Coxeter groups in Section \ref{sect3}, for which I am responsible, and the proof of Theorem \ref{thma52} presented here is due to the author.

\end{document}